\documentclass[11pt]{article}

\usepackage{fullpage}
\usepackage{amsthm}
\usepackage{amssymb}
\usepackage{amsmath}
\usepackage{graphicx}
\usepackage{tikz}
\usepackage{mathrsfs}
\usepackage{float}
\usepackage{makecell}
\usepackage{amscd}
\usepackage{multirow}
\usepackage{cite}
\usepackage{verbatim}
\usepackage{hhline}
\usepackage{comment}

\newcommand{\Z}{\mathbb{Z}}

\newcommand{\Q}{\mathbb{Q}}

\renewcommand{\Im}{\text{Im}}

\renewcommand{\epsilon}{\varepsilon}

\DeclareMathOperator{\rank}{rank}

\DeclareMathOperator{\coker}{coker}
\DeclareMathOperator{\supp}{supp}
\usepackage{relsize}
\newtheorem{theorem}{Theorem}
\newtheorem{lemma}[theorem]{Lemma}

\newtheorem{proposition}[theorem]{Proposition}

\newtheorem{claim}[theorem]{Claim}
\theoremstyle{definition}
\newtheorem{conjecture}[theorem]{Conjecture}
\newtheorem{definition}[theorem]{Definition}

\title{Abelian groups from random hypergraphs}
\author{Andrew Newman\thanks{Carnegie Mellon University}}
\begin{document}
\maketitle
\abstract
For a $k$-uniform hypergraph $\mathcal{H}$ on vertex set $\{1, ..., n\}$ we associate a particular signed incidence matrix $M(\mathcal{H})$ over the integers. For $\mathcal{H} \sim \mathcal{H}_k(n, p)$ an Erd\H{o}s--R\'{e}nyi random $k$-uniform hypergraph, $\coker(M(\mathcal{H}))$ is then a model for random abelian groups. Motivated by conjectures from the study of random simplicial complexes we show that for $p = \omega(1/n^{k - 1})$, $\coker(M(\mathcal{H}))$ is torsion-free.
\section{Introduction}
Let $\mathcal{H}$ be a $k$-uniform hypergraph on vertex set $[n]$. We associate to $\mathcal{H}$ a matrix $M := M(\mathcal{H})$ with rows indexed by the vertices and columns indexed by the edges of $\mathcal{H}$. For $e = \{v_1, v_2, ..., v_k\}$ each $v_i \in [n]$ with $v_1 < v_2 < \cdots < v_k$, $M(v_i, e) := (-1)^{i + 1}$ for $1 \leq i \leq k$ and $M(v, e) := 0$ for $v \notin e$. Our primary object of study is the cokernel of $M(\mathcal{H})$ when $\mathcal{H}$ is a random hypergraph. 

The motivation for this model comes from a question about random simplicial complexes. Recall that the Linial--Meshulam--Wallach model, introduced in \cite{LM, MW}, denoted $Y_d(n, p)$ is the probability space on $d$-dimensional simplicial complexes on $n$ vertices with complete $(d - 1)$-skeleton sampled by including each $d$-face independently with probability $p$. Observe that $Y_1(n, p)$ is the Erd\H{o}s--R\'enyi random graph model. Since the Linial--Meshulam--Wallach model was first introduced much of the research has been to establish thresholds for topological properties that generalize thresholds for graph properties. One of the most important thresholds in $Y_d(n, p)$ is the one-sided sharp threshold for nonvanishing of the $d$th homology group due to Aronshtam, Linial, and Peled \cite{AL, LP}. Namely for each $d \geq 2$, there is an explicit constant $c_d$ defined in \cite{AL} so that if $p = c/n$ with $c < c_d$ then the $d$th homology group of $Y \sim Y_d(n, p)$ with rational coefficients is generated by Poisson-distributed embedded copies of the $(d + 1)$-simplex boundary in $Y_d(n, p)$ \cite{LP}.  On the other hand for $c > c_d$, \cite{AL} shows that with high probability $Y \sim Y_d(n, p)$ has $d$th Betti number of order $\Theta(n^d)$ with high probability. 

While we don't yet have a good understanding of what happens inside the critical window for $p$ near $c_d/n$, experiments conducted by Kahle, Lutz, Newman, and Parsons \cite{KLNP} demonstrate strongly that the homology within the critical window is very interesting. More specifically, the experiments conducted in \cite{KLNP} witness a \emph{torsion burst} in the $(d - 1)$st homology group within the critical window. An instance of the experiment in \cite{KLNP} starts with the complete graph on $n$ vertices, and a 2-complex is constructed from this graph by adding triangles one at a time in random order. That is, at each step one picks uniformly at random a triangle on the ground set that is not already included in the complex and adds it.  In this way we turn $Y_2(n, p)$ into a stochastic process and can study how the two interesting homology groups evolve. Early on, most of the time a new triangle is added, the free rank of $H_1$ drops by one, but occasionally a tetrahedron boundary is completed and so the rank of $H_2$ increases by one instead. Moreover, at the early stages $H_1$ is torsion-free. However, right around the critical density for the phase transition in $H_2$ established in \cite{AL, LP}, one observes large torsion in the first homology group. In a particular instance shown as Table 1 in \cite{KLNP} with $n = 75$, when the 2470th triangle is added, torsion appears in the first homology group. At that point in this experiment the first homology group was $\Z^{235} \times \Z/2\Z$.  What is even more interesting is that when the 2475th triangle is added in this particular run the torsion part of the first homology group has order larger than $10^{26}$, but when one more triangle is added this torsion drops to $\Z/2\Z$. The details of which groups appear can be found in Table 1 of \cite{KLNP}. Once this torsion is gone, we have apparently passed the homology threshold of \cite{AL, LP} and the rank of the second homology group starts to grow quickly.

As the torsion in the $(d - 1)$st homology group of a simplicial complex comes from the cokernel of the $d$th boundary matrix, it seemed reasonable when conducting the experiments for \cite{KLNP} to try and see if the same torsion burst behavior occurred in a random matrix model. The $d$th boundary matrix of a simplicial complex (when choosing orientations from an ordering on the vertices) is a matrix in which each column has exactly $(d + 1)$ nonzero entries and those nonzero entries alternate as $1, -1, 1, -1, ..., (-1)^{d + 1}$. From this perspective when working on \cite{KLNP} it made sense to run experiments to see what happens in a model of random matrices that have this structure (but not the far more restrictive structure coming from the geometry of a simplicial complex), i.e. the matrices described above as $M(\mathcal{H})$ for $\mathcal{H}$ a $(d + 1)$-uniform hypergraph on $[n]$. These experiments turned out to exhibit the same behavior as $Y_d(n, p)$. 

Random matrices with a fixed number of nonzero entries in each row have been considered in the past and \cite{LP} in particular mentions their similarity with $Y_d(n, p)$. For example Pittel and Sorkin \cite{PittelSorkin} consider this same model for random matrices we describe here over $\Z/2\Z$ in studying random instances of $k$-XORSAT and Cooper, Frieze, and Pegden \cite{CooperFriezePegden} prove that a random matrix over $\Z/2\Z$ with exactly $k$ 1's in each column exhibit a phase transition in the rank. The phase transitions described in \cite{PittelSorkin, CooperFriezePegden} are perfectly analogous to the phase transition in homology of $Y_d(n, p)$ established by \cite{AL, LP} in quite a strong sense. In $Y_d(n, c/n)$ with $c$ constant, if $c$ is large enough that the average degree of a $(d - 1)$-dimensional face exceeds $c_d$ then $Y \sim Y_d(n, c/n)$ asymptotically almost surely has nonvanishing $d$th homology while for $c$ small enough that the average degree of $(d - 1)$-dimensional face is smaller than $c_d$ then $d$-homology is generated by $(d + 1)$-simplex boundaries. For this same $c_d$, as is pointed out in \cite{LP}, Pittel and Sorkin show that a random $n \times m$ matrix over $\Z/2\Z$ with exactly $d + 1$ 1's in each column will have nontrivial kernel when the average number of 1's in each \emph{row} exceeds $c_d$ and trivial kernel when this average is below $c_d$. Work of Cooper, Frieze, and Pegden \cite{CooperFriezePegden} refines this result to describe the asymptotic rank of the random $\Z/2\Z$ matrix on either side of the phase transition. It is at the phase transition of \cite{PittelSorkin} that experiments witness a torsion burst in our random matrix model. 

As with the experiments for \cite{KLNP} to see this torsion burst it seems necessary to view the cokernel of $M(\mathcal{H})$ for $\mathcal{H}$ a random hypergraph as a stochastic process. That is we add columns one at a time and compute the cokernel at each step. The results of a sample run are shown in Table \ref{tbl:SampleRun}. The experiment was carried out using GAP \cite{GAP}.

\begin{table}[h]
\centering
\begin{tabular}{c | c}
Number of columns & Cokernel \\
\hline
0 & $\Z^{100}$ \\
1 & $\Z^{99}$ \\
$\vdots$ & $\vdots$ \\
94 & $\Z^6$ \\
95 & $\Z^5 \times \Z/2\Z$ \\
96 & $\Z^4 \times \Z/6\Z$ \\
97 & $\Z^3 \times \Z/894\Z$\\
98 & $\Z^3 \times \Z/3\Z$\\
99 & $\Z^3 \times \Z/3\Z$ \\
100 & $\Z^3$ \\
$\vdots$ & $\vdots$ \\
125 & $\Z$ \\
126 & $0$ \\
\end{tabular}
\caption{Sample run of the random abelian group process with $n = 100$ and $k = 3$} \label{tbl:SampleRun}
\end{table}

Proving the existence of the torsion burst in $Y_d(n, p)$ seems to be quite a difficult problem without obvious tools to approach it. More tractable perhaps is proving its uniqueness. This is formulated as a conjecture of \L uczak and Peled in \cite{LP2}.
\begin{conjecture}[\L uczak and Peled \cite{LP2}]\label{LuczakPeledConjecture}
For every $d \geq 2$ and $p = p(n)$ such that $|np - c_d|$ is bounded away from 0, $H_{d - 1}(Y_d(n, p))$ is torsion-free asymptotically almost surely.
\end{conjecture}

We point out that for $Y \sim Y_d(n, p)$, $H_{d - 1}(Y)$ a.a.s. does not vanish until $p = \frac{d \log n}{n}$. This result was proved for fixed field coefficients in \cite{LM, MW}, and for integer coefficients in the $d = 2$ case in \cite{LP2} and in the general case in \cite{NP}.  Thus Conjecture \ref{LuczakPeledConjecture} would provide a probability regime where $H_{d - 1}(Y)$ is nonvanishing and torsion-free. 

In light of this conjecture and the comparison between random matrices over $\Z/2\Z$ with $(d + 1)$ 1's in every row and $Y_d(n, p)$, we prove the following as our main theorem.
\begin{theorem}\label{maintheorem}
For any $k \geq 3$, if $p = \omega(\frac{1}{n^{k - 1}})$ then the cokernel of $M(\mathcal{H})$ for $\mathcal{H} \sim \mathcal{H}_k(n, p)$ asymptotically almost surely is torsion-free. 
\end{theorem}

We remark also that the absence of torsion far below the phase transition, i.e. for $p = o(1/n^{k - 1})$, follows from facts about the \emph{2-core} of $\mathcal{H}_k(n, p)$. Recall that the 2-core of a hypergraph $\mathcal{H}$ is the hypergraph $\mathcal{H'}$ obtained by successively deleting vertices (and hyperedges containing them) belonging to fewer than two hyperedges. Note that if $\mathcal{H}$ is a $k$-uniform hypergraph and $v \in \mathcal{H}$ is contained in exactly one hyperedge $e$ then deleting $v$ and $e$ from $\mathcal{H}$ does not change $\coker(M(\mathcal{H}))$. One the other hand if $v \in \mathcal{H}$ does not belong to any hyperedges at all then deleting $v$ from $\mathcal{H}$ removes a $\Z$ factor from $\coker(M(\mathcal{H}))$. Therefore the torsion part of $\coker(M(\mathcal{H}))$ is the torsion part of $\coker(M(\mathcal{H}'))$ where $\mathcal{H}'$ is the 2-core of $\mathcal{H}$. 

One of the main results of a paper of Molloy \cite{Molloy} establishes the sharp phase transition for the property that $\mathcal{H} \sim \mathcal{H}_k(n, p)$ for $k \geq 3$ has a nontrivial 2-core to be $\gamma_k/n^{k - 1}$ for an explicit constant $\gamma_k$. In fact the threshold that he establishes for this property is fundamentally the same as the threshold for $d$-collapsibility in the Linial--Meshulam--Wallach model established by \cite{ALLM, AL2} furthering the analogy between $Y_d(n, c/n)$ and $M(\mathcal{H}_k(n, c/n^{k - 1}))$. Because of Molloy's result about the 2-core and the connection between the 2-core and the torsion part of the cokernel we immediately have the following.
\begin{proposition}\label{sparseside}
For $k \geq 3$, if $p = o(\frac{1}{n^{k - 1}})$ then the cokernel of $M(\mathcal{H})$ if $\mathcal{H} \sim \mathcal{H}_k(n, p)$ asymptotically almost surely is torsion-free. 
\end{proposition}

Also because of the comparison between $Y_d(n, p)$ and $M(\mathcal{H}_k(n, p))$ we prove the following analogue to the homology vanishing threshold in our random matrix model. This proof will follow immediately from Theorem \ref{maintheorem} and a result about the $\Z/2\Z$ version of the random matrix model from \cite{CooperFriezePegden}. 
\begin{theorem}\label{GroupVanishes}
For any $k \geq 3$, if $p = \frac{c \log n}{n^{k - 1}}$ then for $c < k!$ the cokernel of $M(\mathcal{H})$ for $\mathcal{H} \sim \mathcal{H}_k(n, p)$ asymptotically almost surely has free rank at least 1 when $k$ is odd and at least 2 when $k$ is even, while for $c > k!$ the cokernel of $M(\mathcal{H})$ for $\mathcal{H} \sim \mathcal{H}_k(n, p)$ asymptotically almost surely is $0$ if $k$ is odd and is $\Z$ if $k$ is even.
\end{theorem}
Note that Theorem \ref{maintheorem} and Theorem \ref{GroupVanishes} implies that there is a probability regime where $\coker(M(\mathcal{H}))$ is torsion-free, but also nontrivial.

\section{Outline of the proof of Theorem \ref{maintheorem}}
Since the cokernel of an $n \times m$ matrix $M$ is $\Z^n/\Im(M)$, a torsion element for $\coker(M)$ is an integer vector $w$ so that $w$ is not in the image of $M$ but $tw$ is in the image of $M$ for some $t \geq 2$. Obviously it suffices to rule out $q$-torsion over all primes $q$ i.e. to show that for all primes $q$ there is never an integer vector $w$ so that $qw \in \Im(M)$ but $w \notin \Im(M)$, however there is an important subtlety. We need to show that in the probability regime considered we can bound the probability that there is $q$-torsion \emph{simultaneously over all primes $q$}. It would not be enough to fix $q$ and show that the probability of $q$-torsion is $o(1)$. Such an approach would leave open the possibility that there is $q(n)$-torsion in our random model for some \emph{sequence} of primes $q(n)$ growing with $n$. 

We first state the key lemmas to sketch out the arguments. It turns out to be easier to rule out torsion in $\coker(M^T)$ rather than directly ruling out torsion for $\coker(M)$. This is similar to how \cite{LM, MW, LP2, NP} prove cohomology vanishing theorems rather than directly proving homology vanishing theorems. This formulation is equivalent since the torsion coefficients of an integer matrix come from its Smith normal form, the torsion part of $\coker(M^T)$ is the same as the torsion part of $\coker(M)$ for any matrix. 

The idea of the proof is to show that with high probability there are no $w$, $v$, $q$ so that $M^Tv = qw$ with $q$ prime, $v \notin (q\Z)^n$, and $w \notin \Im(M^T)$. The argument splits depending on the size of the support of $v$. The first key lemma has to do with the case that the support of $v$ is small. Here we introduce the notation $M_S$ for $S$ a subset of the columns of $M$ to be the submatrix obtained by restricting $M$ to the columns belonging to $S$. 
\begin{lemma}\label{newSmallCocycles}
Fix $\delta \in (0, 1)$ then for $\mathcal{H} \sim \mathcal{H}_k(n, p)$ with $p = \omega(1/n^{k - 1})$ asymptotically almost surely for every set $S \subseteq [n]$ with $|S| < \delta n$,  $\rank_{\Q}(M^T_S) = \rank_{\Z/q\Z}(M^T_S)$ for every prime $q$ where $M = M(\mathcal{H})$. 
\end{lemma}

For the case that the support of $v$ is large the argument splits depending on the parity of $k$. For the case that $k$ is odd the argument is a bit more straightforward. The issue with $k$ is even is that the row sum along every row of $M^T$ is zero, so the all-ones vector is always in the kernel of $M^T$ which makes the argument slightly more difficult. The lemma to handle the odd case is the following. Note that here we use $\mathcal{H}_k(n, m)$ the uniform distribution on $k$-uniform hypergraphs with exactly $m$ hyperedges.

\begin{lemma}\label{LargeCocyclesLemmaOdd}
For any $k \geq 3$ odd, $\delta > 0$, and $c = c(k, \delta)$ a sufficiently large constant there exists $C = C(k, c, \delta) > \log(k)$ so that for any prime $q$ the probability that $\ker_{\Z/q\Z}(M^T)$ for $M = M(\mathcal{H})$, $\mathcal{H}_k(n, cn)$ contains a vector of support size at least $\delta n$ is at most $e^{-Cn}$.
\end{lemma}

Note here that the $C$ in the statement does not depend on $q$. This is important because the following lemma tells us that we only have to consider exponentially many primes. This lemma is well known, and it seem that the earliest reference to it is a paper of cite Soul\'e \cite{Soule}. It appears in papers about vanishing homology theorems for random complexes. In particular it appears with proof as Claim 2 of \cite{NP}.

\begin{lemma}\label{HowManyPrimes}
If $M$ is an integer matrix so that each column of $M$ has Euclidean norm at most $t$ then the torsion part of $\coker(M)$ has size at most $t^n$. 
\end{lemma}

With Lemmas \ref{newSmallCocycles}, \ref{LargeCocyclesLemmaOdd}, and \ref{HowManyPrimes} we can prove Theorem \ref{maintheorem} for $k$ odd.
\begin{proof}[Proof of Theorem \ref{maintheorem} for odd $k$]
For $p = \omega \left(1/n^{k - 1} \right)$ we bound the probability that $M := M(\mathcal{H})$ has cokernel with torsion when $\mathcal{H} \sim \mathcal{H}_k(n, p)$. If $\coker(M)$ has torsion then it has $q$-torsion for some prime $q \leq \sqrt{k}^n$. Moreover if $M$ has $q$-torsion then so does $\coker(M^T)$, so in that case there exists a vector $v$ so that $M^Tv = qw$ for $w$ an integer vector not in the image of $M^T$ and $v \notin (q\Z)^n$. In this case $v \pmod q$ is a nontrivial vector in $\ker_{\Z/q\Z}(M^T)$. By Lemma \ref{LargeCocyclesLemmaOdd} and a union bound over all $\sqrt{k}^n$ primes under consideration we see that with high probability there is no vector in $\ker_{\Z/q\Z}(M^T)$ with support size larger than $n/2$ for any prime $q \leq \sqrt{k}^n$. Indeed letting $T_q$ denote the event that there is a vector $v \in \ker_{\Z/q\Z}(M^T)$ with support size at least $n/2$ and taking union bound over all primes smaller than $\sqrt{k}^n$ we find the probability that there is $v \in \ker_{\Z/q\Z}(M^T)$ of support size at least $n/2$ for some prime $q$ is at most
\begin{eqnarray*}
\Pr\left(\bigcup_{q \leq \sqrt{k}^n, q \text{ prime }} T_q\right) &\leq& \Pr\left(\bigcup_{q \leq \sqrt{k}^n, q \text{ prime }} T_q \middle| |E(\mathcal{H})| > cn\right) + \Pr(|E(\mathcal{H})| \leq cn)\\
&\leq& \sum_{q \leq \sqrt{k}^n, q \text{ prime }} \Pr(T_q | |E(H)| > cn) + \Pr(|E(H)| \leq cn)
\end{eqnarray*}
where $c$ is some sufficiently large constant for the assumptions of Lemma \ref{LargeCocyclesLemmaOdd} to hold. By a standard coupling argument between $\mathcal{H}_k(n, cn)$ and $\mathcal{H}_{k}(n, p)$ for $p = \omega(1/n^{k - 1})$, we apply Lemma \ref{LargeCocyclesLemmaOdd} to conclude that each term in the sum over at most $\sqrt{k}^n$ primes is at most $\exp(-Cn)$ for $C > \log(k)$. Thus the large sum is at most $\exp(-n\log(k)/2)$ and the last term is $o(1)$ since the number of hyperedges of $\mathcal{H}$ is distributed as a binomial with mean $\omega(n)$. 

  So we can assume that the vector $v$ with $M^Tv = qw$ has support size (over $\Z/q\Z$) at most $n/2$. We know however that with high probability $M^T$ satisfies the assumption of Lemma \ref{newSmallCocycles} with $\delta = 1/2$. In this case then taking $S$ to be the support of $v$ over $\Z/q\Z$, we see that $\rank_{\Q}((M^T)_S) = \rank_{\Z/q\Z}((M^T)_S)$. Therefore $\coker((M^T)_S)$ is $q$-torsion-free. As $S$ is the support of $v$ over $\Z/q\Z$,  $v = v_1 + qv_2$ where $\supp(v_1) = S$, we have that $qw = M^T(v_1 + qv_2) \in \coker((M^T))$, so $M^T(v_1) = qw_1$ for some vector $w_1$. As $v_1$ is supported on $S$ and $\coker((M^T)_S)$ is torsion-free there is some integer vector $u$ supported on $S$ with $M^Tu = w_1$ and so $M^T(u + v_2) = w$ contradicting the choice of $w$ as a torsion element of $\coker(M^T)$.
\end{proof}
The proof when $k$ is even makes use of a definition we introduce later, so we save the proof for the even case for the end.

\section{Vectors with small support}
Here we prove Lemma \ref{newSmallCocycles}. Toward that goal we start with the following definition.
\begin{definition}
Let $M$ be an integer matrix, a subset $S$ of the columns of $M$ is said to be a \emph{torsion cocycle} provided that there exists a prime $q$ so that $\rank_{\Q}(M_S) > \rank_{\Z/q\Z}(M_S)$. A \emph{minimal torsion cocycle} is an inclusion-minimal set of columns $S$ which is a torsion cocycle.
\end{definition}
The name torsion cocycle comes from the fact that our proof of the main theorem here is an adaptation of the cocycle counting argument from proofs of homology vanishing theorems in \cite{LM, MW, NP}. The following claim about minimal torsion cocycles is easy to prove and sets up a sub-hypergraph inclusion problem to rule out small torsion cocycles and prove Lemma \ref{newSmallCocycles}. 
\begin{claim}\label{RankClaim}
If $S$ is a minimal torsion cocycle of a matrix $M$, then $\rank_{\Q}(M_S) = |S|$. Moreover $M_S$ cannot have a row which has $|S| - 1$ entries equal to zero and the remaining entry equal to $1$ or to $-1$.
\end{claim}
\begin{proof}
Suppose that $\rank_{\Q}(M_S) < |S|$. Let $U \subsetneq S$ be a maximal collection of linearly independent columns of $M_S$ over $\Q$. Since $S$ is a torsion cocycle there exists a prime $q$ so that $\rank_{\Q}(M_S) > \rank_{\Z/q/\Z}(M_S)$. By minimality of $S$, $\rank_{\Q}(M_U) = \rank_{\Z/q\Z}(M_U)$. On the other hand we have
\begin{eqnarray*}
\rank_{\Q}(M_U) &=& \rank_{\Q}(M_S) \\
&>& \rank_{\Z/q\Z}(M_S) \\
&\geq& \rank_{\Z/q\Z}(M_U).
\end{eqnarray*}
Thus, we reach a contradiction and finish the proof of the first part of the claim.

For the second part we observe that if $z$ is a row of $M_S$ so that the $i$th entry of $z$ is $\pm 1$ and all other entries of $z$ are $0$ then clearly over any field column $i$ of $M_S$ is outside the span of all other columns of $M_S$. Thus deleting $i$ from $S$ drops the rank of $M_S$ over \emph{any} field by 1. Thus $\rank_{\Q}(M_{S \setminus \{i\}}) > \rank_{\Z/q\Z}(M_{S \setminus \{i\}})$ when $\rank_{\Q}(M_S) > \rank_{\Z/q\Z}(M_S)$ as both ranks drop by one. However this contradicts minimality.
\end{proof}
By the claim we have that any set of columns $S \subseteq [n]$ of $M^T$ which form a minimal torsion cocycle there must be at least $|S|$ nonzero rows in $M^T_S$ in order for the rank of $M^T_S$ over $\Q$ to be at least $|S|$. As the columns of $M^T$ are indexed by vertices of $\mathcal{H}$ and the rows are indexed by hyperedges of $\mathcal{H}$, for $S$ the set of $k$ vertices corresponding to the minimal torsion cocycle we have by the two parts of Claim \ref{RankClaim}:
\begin{enumerate}
\item There are at least $k$ rows of $M^T$ with support intersecting $S$. That is there are at least $k$ hyperedges in the random hypergraph which involve at least one vertex of $S$.
\item There are no rows of $M^T$ which intersect $S$ in exactly one place. In terms of the underlying hypergraph no hyperedge contains exactly one vertex of $S$.
\end{enumerate}

We use a linearity of expectation argument to bound the probability that $\mathcal{H} \sim \mathcal{H}_k(n, p)$ for $p = \omega(1/n^{k - 1})$ has a small subset of vertices $S$ which satisfies conditions (1) and (2) above.

\begin{lemma}\label{SmallCocyclesLemma}
Fix $\delta \in (0, 1)$. With high probability $\mathcal{H} \sim \mathcal{H}_k(n, p)$ for $p = \omega(1/n^{k - 1})$ does not contain a subset of vertices $W$ of size at most $\delta n$ so that both of the following hold:
\begin{itemize}
\item No hyperedge intersects $W$ exactly once, and 
\item there are at least $|W|$ hyperedges that intersect $W$. 
\end{itemize}
\end{lemma}
\begin{proof}
We prove this by a linearity of expectation argument counting the expected number of subsets $W$ for which there is no hyperedge that intersects $W$ exactly at a single vertex and there are at least $|W|$ hyperedges that intersect $W$ (necessarily each in at least two vertices). For each $1 \leq t \leq \delta n$ there are $\binom{n}{t}$ ways to pick $W$. From here we have to choose at least $t$ hyperedges that intersect $W$ in at least 2 vertices to be included. The number of hyperedges that intersect $W$ in at least 2 vertices is 
\[\sum_{i = 0}^{k -2} \binom{n - t}{i} \binom{t}{k - i} \leq \sum_{i = 0}^{k - 2} n^i t^{k - i} .\]
And from this set we choose at least $t$ hyperedges to be included. Next we have to guarantee that all hyperedges that intersect $W$ in only one place are excluded. The number of hyperedges that meet $W$ in exactly one vertex is 
\[t \binom{n - t}{k - 1}.\]
Putting this all together the number of expected number sets $W$ that satisfy the conditions we want to avoid is at most 
\begin{eqnarray*}
\sum_{t = 1}^{\delta n} \binom{n}{t} \left( \sum_{i = 0}^{k -2} n^i t^{k - i} \right)^t p^t \exp\left(-p \binom{n - t}{k - 1}\right)^t.
\end{eqnarray*}
Taking $p = f/n^{k - 1}$ with $f := f(n) \rightarrow \infty$ arbitrarily slowly, we have the above is at most 
\begin{eqnarray*}
\sum_{t = 1}^{\delta n} \frac{n^t e^t}{t^t}  \left( \sum_{i = 0}^{k -2} n^i t^{k - i} \right)^t (f/n^{k - 1})^t \exp\left(-(f/n^{k - 1}) \frac{(1 - \delta)^{k - 1}n^{k - 1}}{(k - 1)^{k - 1}}\right)^t.
\end{eqnarray*}
Thus for some constant positive $c_k$ that depends on $k$ and $\delta$ we have that the above sum is at most
\begin{eqnarray*}
\sum_{t = 1}^{\delta n} \left(e n^{2 - k} \sum_{i = 0}^{k - 2} n^i t^{k - i - 1} \right)^t \left(f/n^{k - 1} \right)^t  \exp\left(-f c_k\right)^t &\leq& \sum_{t = 1}^{\delta n} \left(e n^{2 - k} f\sum_{i = 0}^{k - 2} \delta^{k - i - 1}  \right)^t  \exp\left(-f c_k\right)^t \\
&\leq& \sum_{t = 1}^{\delta n} \left( e (k - 1) f e^{-c_k f} \right)^t
\end{eqnarray*}
And this is $o(1)$ as long as $f$ tends to infinity. Thus by Markov's inequality we have the lemma.
\end{proof}

Lemma \ref{newSmallCocycles} follows immediately from Claim \ref{RankClaim} and Lemma \ref{SmallCocyclesLemma}.

\section{Vectors with large support}

We turn our attention now to the proof of Lemma \ref{LargeCocyclesLemmaOdd}. As this proof is about $\mathcal{H}_k(n, m)$ we consider $\mathcal{H}_k(n, m)$ and the corresponding $M(\mathcal{H})$ as a stochastic process obtained by adding the hyperedges of $\binom{[n]}{k}$ one at a time in random order. In order to prove Theorem \ref{LargeCocyclesLemmaOdd} we want to lower bound the probability that the dimension of the kernel of $M^T$ drops as we add the random rows of $M^T$ one at a time. Toward that end we introduce the following definition.
\begin{definition}
Take $V_{k, n}$ to be the set of vectors that can be column vectors of $M(\mathcal{H})$, i.e. $V_{k, n}$ is the set of vectors in $\Z^n$ with exactly $k$ nonzero entries and the nonzero entries alternating as $1,-1, 1, -1, ...$. For $k \geq 3$, $q$ prime, and $v \in (\Z/q\Z)^n$ let $B_q(v)$ denote the set $\{w \in V_{k, n} \mid w \cdot v \neq 0 \text{ over $\Z/q\Z$.}\}$. We observe that for any matrix $M$ with rows in $V_{k, n}$ if $v \in \ker_{\Z/q\Z}(M)$, and we add a row in $V_{k, n}$ to it to create a new matrix $M'$ then $v \in \ker_{\Z/q\Z}(M')$ if and only if the new row is outside of $B_q(v)$. 
\end{definition}
We prove the following statement regarding $B_q(v)$ when $k$ is odd.
\begin{lemma}\label{CoisoperimetryOddK}
For $k$ odd there exists $\gamma := \gamma(k) > 0$ so that for every prime $q$ and every vector $v \in (\Z/q\Z)^n$, $|B_q(v)| \geq \gamma |\supp(v)|^k$. 
\end{lemma}
In order to prove this lemma we introduce the following definition.
\begin{definition}
For $q$ a prime and $\epsilon > 0$ we say that $v \in (\Z/q\Z)^n$ is \emph{$\epsilon$-balanced} provided every nonzero element of $v$ appears at most $\epsilon |\supp(v)|$ times in $v$.
\end{definition}
With the definition of $\epsilon$-balanced the proof of Lemma \ref{CoisoperimetryOddK} splits into two cases.
\begin{proof}[Proof of Lemma \ref{CoisoperimetryOddK}]
For a vector $v \in (\Z/q\Z)^n$ we wish to bound from below the number of $w \in V_{k, n}$ so that $w \cdot v \neq 0$. With foresight into the calculations fix some positive $\epsilon < \frac{(k - 1)!}{k^{k}}$. We consider first the case that $v$ is $\epsilon$-balanced. 

If $v$ is $\epsilon$-balanced then there are at least 
\[\binom{|\supp(v)|}{k} - \binom{|\supp(v)|}{k - 1} \epsilon |\supp(v)|\]
vectors $w$ in $V_{k, n}$ over $\Z/q\Z$ so that $w \cdot v \neq 0$. To see this we restrict only to those vectors $w \in V_{k, n}$ so that $\supp(w) \subseteq \supp(v)$; denote this set by $V_{k, n}^v$. There are clearly (exactly) $\binom{|\supp(v)|}{k}$ vectors in $V_{k, n}^v$. We claim that among these vectors at most $\binom{|\supp(v)|}{k - 1} \epsilon |\supp(v)|$ will be orthogonal over $\Z/q\Z$ to $v$. To construct a vector $w$ in $V_{k, n}^v$ that's orthogonal to $v$ we first pick the first $k - 1$ positions in $\supp(v)$ to be nonzero in $w$. Clearly there are at most $\binom{|\supp(v)|}{k - 1}$ ways to do this. If $w'$ denotes this vector (in $V_{k -1, n}^v$) then if $w' \cdot v = 0$ then there is no way to select the last position of the nonzero entry in $w$ so that $\supp(w) \subseteq \supp(v)$ and $w \cdot v = 0$. On the other hand if $w' \cdot v \neq 0$ then in order for $w \cdot v$ to be zero, the last entry of $w$ must fall into a position where $v$ is the unique nonzero element of $\Z/q\Z$ equal to $-w' \cdot v$. By the $\epsilon$-balanced condition there are at most $\epsilon |\supp(v)|$ choices for this position. Thus $V_{k, n}^v$ has $\binom{|\supp(v)|}{k}$ vectors and at most $\binom{|\supp(v)|}{k - 1} \epsilon |\supp(v)|$ of them are orthogonal to $v$. 

Suppose on the other hand that $v$ is not $\epsilon$-balanced. Then there is some subset $S$ of $\supp(v)$ with $|S| \geq \epsilon|\supp(v)|$ so that $v$ restricted to $S$ is a constant, nonzero function in $\Z/q\Z$. Then any vector in $V_{k, n}$ with support contained in $S$ cannot be orthogonal to $v$ as the dot product of $v$ with such a vector is simply $x - x + x - x \cdot + x  = x$. Thus there are at least 
\[\binom{\epsilon|\supp(v)|}{k}\]
vectors in $V_{k, n}$ that are not orthogonal to $v$ when $v$ is not $\epsilon$ balanced. Thus we have the claim with 
\[\gamma = \min\left\{\frac{1}{k^k} - \frac{\epsilon}{(k - 1)!}, \frac{\epsilon^k}{k^k} \right\}.\]
\end{proof}
Now we prove Lemma \ref{LargeCocyclesLemmaOdd}.
\begin{proof}[Proof of Lemma \ref{LargeCocyclesLemmaOdd}]
Consider the stochastic process to build $M^T$ one row at a time. Let $M^T_i$ denote the $i \times n$ matrix at step $i$ and let $Z_i$ be the random variable $\dim(\ker_{\Z/q\Z}(M^T_i))$ with the convention that $Z_0 = n$. If $M^T_i$ has a vector of support size at least $\delta n$ in its kernel over $\Z/q\Z$ then the probability that $K_{i + 1} < K_i$ is at least \[\frac{\gamma (\delta n)^{k}}{\binom{n}{k}} \geq k! \delta^{k} \gamma,\]
where $\gamma$ is as in Lemma \ref{CoisoperimetryOddK}. Thus the probability that $M^T_{cn}$ for $c$ a large constant has a kernel element over $\Z/q\Z$ of support size at least $\delta n$ is at most the probability that a binomial random variable with $cn$ trials and success probability $k! \delta^k \gamma$ has at most $n$ successes. By Chernoff's bound this is at most $\exp(-Cn)$ where $C$ can be made an arbitrarily large constant by setting $c$ large enough. 
\end{proof}

Now we turn our attention to the even case. Fortunately the relevant definition we'll need is the $\epsilon$-balanced definition we've already introduced. The analogue of Lemma \ref{LargeCocyclesLemmaOdd} is the following:
\begin{lemma}\label{LargeCocyclesLemmaEven}
For any $k \geq 4$ even, $\delta > 0$, and $0 < \epsilon < \frac{(k -1)!}{k^k}$ and $c = c(k, \delta, \epsilon)$ sufficiently large there exists $C = C(c) > \log(k)$ so that for any prime $q$ the probability that $\ker_{\Z/q\Z}(M^T)$ for $M = M(\mathcal{H})$, $\mathcal{H} \sim \mathcal{H}_k(n, cn)$ contains an $\epsilon$-balanced vector of support size at least $\delta n$ is at most $e^{-Cn}$.
\end{lemma}
The proof of this proceeds from the following lemma exactly as the proof of Lemma \ref{LargeCocyclesLemmaOdd}.
\begin{lemma}\label{CoisoperimetryEvenK}
For $k \geq 4$ even and $0 < \epsilon < \frac{(k - 1)!}{k^k}$ there exists $\gamma = \gamma(k, \epsilon)$ so that for any prime $q$ and every $\epsilon$-balanced vector $v \in (\Z/q\Z)^n$, $|B_q(v)| \geq \gamma |\supp(v)|^k$.
\end{lemma}
The proof of Lemma \ref{CoisoperimetryEvenK} is exactly the $\epsilon$-balanced case of the proof of Lemma \ref{CoisoperimetryOddK}. Now we are ready to prove Theorem \ref{maintheorem} in the case that $k$ is even.
\begin{proof}[Proof of Theorem \ref{maintheorem} for even $k$]
As in the odd $k$ case it suffices to show that with high probability there are no $v$, $w$, and $q$ so that $M^Tv = qw$, $w \notin \Im(M^T)$ and $v \notin (q\Z)^n$. Fix a positive $\epsilon < (k - 1)!/k^k$ and take $\delta = (1 + \epsilon)^{-1}$. By Lemma \ref{SmallCocyclesLemma} the probability that $M^T$ has a torsion cocycle of size at most $\delta n$ is $o(1)$. If $M^Tv = qw$ with $S := \supp(v)$ of size at most $\delta n$ then we use the fact that with high probability $\coker((M^T)_S)$ is torsion-free to conclude that $w \in \Im(M^T)$.  

For vectors of large support, by a similar union bound to the odd $k$ case we can use Lemma \ref{LargeCocyclesLemmaEven} and Lemma \ref{HowManyPrimes} to rule out $\epsilon$-balanced vector $v \in \Z/q\Z$ with $|\supp(v)| \geq \delta n$ and $v \in \ker_{\Z/q\Z}(M^T)$. The only case that's different is if $v$ is not $\epsilon$-balanced but still has large support. If $v$ is not $\epsilon$-balanced then some $x \in \Z/q\Z \setminus \{0\}$ appears more than $\epsilon \delta n$ times in $v \pmod q$. In this case then $x\textbf{1} - v$ satisfies $|\supp(x\textbf{1} - v)| \leq n - \epsilon \delta n = \delta n$, and $M^T(x \textbf{1} - v) = -qw$, but then $S = \supp(x \textbf{1} - v)$ has size at most $\delta n$ so we appeal again to the fact that $\coker((M^T)_S)$ has no $q$-torsion to conclude that $w \in \Im(M^T)$. 

\end{proof}

\section{Proof of Theorem \ref{GroupVanishes}}
For the proof of Theorem \ref{GroupVanishes}, we use Theorem \ref{maintheorem} as well as a theorem of Cooper, Frieze, and Pegden \cite{CooperFriezePegden}. In \cite{CooperFriezePegden}, the authors consider the same model that we consider except that the matrices are taken over $\Z/2\Z$ rather than over $\Z$. Theorem 1.3 from \cite{CooperFriezePegden} and the usual coupling argument between $\mathcal{H}_k(n, p)$ and $\mathcal{H}_k(n, m)$ immediately give the following.
\begin{theorem}[Corollary to Theorem 1.3 of \cite{CooperFriezePegden}]\label{CFPResult}
Let $n^* := n$, if $k$ is odd, and $n^* := n - 1$, if $k$ is even. For $p = \frac{c \log n}{n^{k - 1}}$ the following hold for $\mathcal{H}_k(n, p)$.
\begin{itemize}
\item If $c < k!$ then with high probability $\rank_{\Z/2\Z}(M(\mathcal{H}))$ is smaller than $n^*$.
\item If $c > k!$ then with high probability $\rank_{\Z/2\Z}(M(\mathcal{H}))$ is equal to $n^*$.
\end{itemize}
\end{theorem}
We can now give a short proof of Theorem \ref{GroupVanishes}.
\begin{proof}[Proof of Theorem \ref{GroupVanishes}]
For any matrix $m \times n$ matrix $M$ the cokernel of $M$ can be read off of the elementary divisors of $M$, so in particular the free part of $\coker(M)$ is $\Z^{n - \rank_{\mathbb{Q}}(M)}$. Moreover the inequality $\rank_{\mathbb{Q}}(M) \geq \rank_{\Z/2\Z}(M)$ for $M$ an integer matrix always holds since $Mv =0$ over $\Q$ implies that $Mv = 0$ over $\Z/2\Z$. For $c < k!$ by Theorem \ref{CFPResult}, with high probability the free part of $\coker(M(\mathcal{H}))$ has rank larger than $n - n^*$. Thus we have the first part of Theorem \ref{GroupVanishes}.

For $c > k!$, we know that $\rank_{\Q}(M(\mathcal{H})) \geq \rank_{\Z/2\Z}(M(\mathcal{H})) = n^*$ asymptotically almost surely, and trivially $\rank_{\Q}(M(\mathcal{H})) \leq n^*$. Thus $\rank_{\Q}(M(\mathcal{H})) = n^*$ in this regime. Thus for $k$ odd, with high probability $\coker(M(\mathcal{H}))$ is finite, and for $k$ even, with high probability $\coker(M(\mathcal{H})) = \Z \oplus A$ for some finite abelian group $A$. However we know by Theorem \ref{maintheorem} that $\coker(M(\mathcal{H}))$ is torsion-free, so we conclude in the odd $k$ case we are left with the trivial group and in the even $k$ case we are left with $\Z$.
\end{proof}
\section{Concluding remarks}
The main result here was motivated by the conjecture of \L uczak and Peled about random simplicial complexes described earlier.  Based on Theorem \ref{maintheorem} it would seem plausible that an adaptation of the methods here could be used to show that $p = \omega(1/n)$ implies that $H_{d - 1}(Y)$ for $Y \sim Y_d(n, p)$ is torsion-free with high probability.

In another direction, it would be interesting to extend the model for $M(\mathcal{H})$ here to allow for the nonzero entries to come from a sequence other than $1, -1, 1, -1, ...$. For example, if we instead defined $M(\mathcal{H})$ by having the nonzero entries all be equal to 1, then clearly $\textbf{1}$ is an element of the kernel of $M^T$ for primes that divide $k$, but not over $\Q$, and so there would always be $k$-torsion even when $p = 1$. Nonetheless, it seems likely that in this version of the model one could show that for $p = \omega(1/n^{k - 1})$ with high probability there is no torsion other than the torsion which exists in $\coker(M(\mathcal{H}))$ for $\mathcal{H}$ the complete $k$-uniform hypergraph on $n$, which would be $\Z/k\Z$.

\section*{Acknowledgment}
The author thanks Elliot Paquette for helpful discussions at the beginning of the research for this article and for comments on an earlier draft.
\bibliography{ResearchBibliography}
\bibliographystyle{amsplain}
\end{document}